\documentclass{amsart}

\usepackage{amsfonts,amssymb,amsmath,amsthm}
\usepackage{url}
\usepackage{enumerate}

\usepackage{verbatim}
\usepackage{color}
\usepackage{enumerate}
\theoremstyle{plain}
\newtheorem{theorem}{Theorem}[section]

\newtheorem{proposition}[theorem]{Proposition}
\newtheorem{lemma}[theorem]{Lemma}

\theoremstyle{definition}
\newtheorem{definition}[theorem]{Definition}

\theoremstyle{remark}
\newtheorem{remark}[theorem]{Remark}
\newtheorem{example}[theorem]{Example}

\numberwithin{equation}{section}

\def\cocoa{{\hbox{\rm C\kern-.13em o\kern-.07em C\kern-.13em o\kern-.15em A}}}

\title[When the positivity of the $h$-vector implies Cohen-Macaulayness]
{When the positivity of the $h$-vector implies the Cohen-Macaulay property}

\author[F. Cioffi]{Francesca Cioffi}
\email{cioffifr@unina.it, digennar@unina.it}
\author[R. Di Gennaro]{Roberta Di Gennaro}

\address{Dip. di Matematica e Applicazioni, Universit\`a di Napoli ``Federico II", 80126 Napoli, Italy
}


\begin{document}

%\maketitle

\begin{abstract}
We study relations between the Cohen-Macaulay property and the positivity of $h$-vectors, showing that these two conditions are equivalent for those locally Cohen-Macaulay equidimensional closed projective subschemes $X$, which are close to a complete intersection $Y$ (of the same codimension) in terms of the difference between the degrees. More precisely, let $X\subset \mathbb P^n_K$ ($n\geq 4$) be contained in $Y$, either of codimension two with $deg(Y)-deg(X)\leq 5$ or of codimension $\geq 3$ with $deg(Y)-deg(X)\leq 3$. Over a field $K$ of characteristic $0$, we prove that $X$ is arithmetically Cohen-Macaulay if and only if its $h$-vector is positive, improving results of a previous work. We show that this equivalence holds also for space curves $C$ with $deg(Y)-deg(C)\leq 5$ in every characteristic $ch(K)\neq 2$. Moreover, we find other classes of subschemes for which the positivity of the $h$-vector implies the Cohen-Macaulay property and provide several examples.
\end{abstract}

\maketitle

%%%%%%%%%%%%%%%%%%
%% Introduction %%
%%%%%%%%%%%%%%%%%%

\section*{Introduction}

There is a deep interest in relations between the structure of standard graded algebras or local rings and their Hilbert functions, and there is a vast literature about this subject (see, for instance, \cite{GMR,Va,RTV} and the references therein).

In this paper we consider the question about when the {\em positivity} of the $h$-vector implies the Cohen-Macaulay property, being the vice versa true and well known. The question we pose here is more subtle than that we faced in \cite{CDG}, i.e.~to ask when the {\em admissibility} of the $h$-vector forces the Cohen-Macaulay property, because the admissibility of the $h$-vector implies its positivity.
We show that our new question makes sense by exhibiting a locally Cohen-Macauly (lCM, for short) curve with positive but not admissible $h$-vector (see Example \ref{ESEMPIO}), where we mean that a curve is a $1$-dimensional closed projective scheme (a lCM curve is also equidimensional, i.e.~its defining ideal is unmixed). 
We find several classes of projective closed subschemes for which our question has a positive answer, improving the results described in \cite{CMR,CDG}.

Our question can be investigated also looking for projective subschemes that have positive $h$-vector although they are not arithmetically Cohen-Macaulay (aCM, for short). This kind of approach is treated also by means of the admissibility of the $h$-vector in \cite{DGM,D,dQR05,CDG,MN}, in different situations, with different approaches, and in the context of symplicial complexes (see \cite{GPSY} and the references therein).

As it is common in the study of the Cohen-Macaulay property, we use general hyperplane sections and, hence, the properties of $0$-dimensional schemes, which are always arithmetically Cohen-Macaulay. Nevertheless, the knowledge alone of $0$-dimensional schemes is not sufficient to give answers to our question. So, we look also at the features of liaison, at the notion of extremal curves and at properties of Borel ideals.

Let $X\subset \mathbb P^n_K$ ($K$ algebraically closed field) be a locally Cohen-Macaulay equidimensional closed projective subscheme, which is contained in a complete intersection $Y$ of the same codimension $c$. In our investigation,
we show that the Cohen-Macaulay property and the positivity of the $h$-vector are equivalent for those subschemes $X$ which are close to $Y$ in terms of the difference between the degrees, precisely in the following cases:
\begin{itemize}
\item[(i)] $n=3$, $c=2$ and $deg(Y)-deg(X)\leq 5$, over a field $K$ of characteristic different from two (Theorem 2.2);
\item[(ii)] $n\geq 4$, $c=2$ and $deg(Y)-deg(X)\leq 5$ over a field $K$ of null characteristic (Theorem 2.3);
\item[(iii)] $n\geq 4$, $c\geq 3$ and $deg(Y)-deg(X)\leq 3$ over a field $K$ of null characteristic (Theorem 3.2).
\end{itemize}
In this perspective, our result for curves is not improvable, due to suitable examples provided by E. D. Davis (see Section 2 of \cite{D}) and to developments of them (see \cite[Examples 4.6 and 4.7 and the Appendix]{CDG}). Nevertheless, there are other classes of subschemes for which the positivity of the $h$-vector implies the Cohen-Macaulay property (see Propositions \ref{varie} and \ref{prop:c3}, and also \cite[Proposition 4.6]{CDG}).

Moreover, by exploiting the Davis' technique, by Proposition \ref{counterexamples} we provide a class of non-aCM but lCM equidimensional surfaces $X\subset\mathbb P^4_K$ with admissible $h$-vector, such that $deg(Y)-deg(X)\geq 10$. The schemes that belong to this class are constructed applying an odd number of liaisons, instead of sequences of basic double links, that are the tools used in \cite{MN} to obtain an analogous construction in even liaison classes.

For the study of codimension $\geq 3$, we compute saturated Borel ideals with a given Hilbert polynomial by the applet \textsc{BorelGenerator} of P. Lella, available at \texttt{www.personalweb.unito.it/paolo.lella/HSC/borelGenerator.html} and based on an algorithm described in \cite{CLMR} and further developed in \cite{L} (an analogous algorithm is described in \cite{MoNa}).

%%%%%%%%%%%%%%%%%%%%%%%%%%%%%%%
%% Setting and basic results %%
%%%%%%%%%%%%%%%%%%%%%%%%%%%%%%%

\section{Setting and basic results}

Let $K$ be an algebrically closed field, $S:=K[x_0,\ldots,x_n]$ be the polynomial ring over $K$ in $n+1$ variables, ${\mathbb P}^n_K = Proj\ S$ the projective space of dimension $n$ over $K$. For a homogeneous ideal $I\subset S$, $I_t$ denotes the $K$-vector space of the homogeneous polynomials of degree $t$ of $I$ and $I_{\leq t}$ the ideal generated by the polynomials of degree $\leq t$ of $I$.
%The integer $\alpha(I):=\min\{t\in \mathbb N \ \vert \ I_t\not=0\}$ is the {\it initial degree} of $I$.
The {\it saturation} of $I$ is $I^{sat}=\{f\in S\ \vert \ \ \forall \ j=0,\ldots,n,\exists \ r \in {\mathbb N} : x_j^r f \in I\}$.
% The ideal $I$ is {\it saturated} if $I^{sat}=I$ and is {\it $m$-saturated} if $(I^{sat})_t = I_t$ for each $t\geq m$.

A finitely generated graded $S$-module $M$ is $m$-{\it regular}
if the $i$-th syzygy module of $M$ is generated in degree
$\leq m+i$, for all $i\geq 0$. The {\it regularity} $reg(M)$ of $M$
is the smallest integer $m$ for which $M$ is $m$-regular.

We use the common notation of ideal sheaf cohomology referring to \cite{H,Mi}. A coherent sheaf $\mathcal F$ on ${\mathbb P}_K^n$ is $m$-{\it regular} if $H^i(\mathcal F(m-i)) = 0$ for all $i>0$. The {\it Castelnuovo-Mumford regularity } (or {\it regularity}) $reg(\mathcal F)$ of $\mathcal F$ is the smallest integer $m$ for
which $\mathcal F$ is $m$-regular.

If $X \subset {\mathbb P}^n_K$ is a closed subscheme, its {\em regularity} $reg(X)$ is defined as the regularity $reg(I)$ of its defining ideal $I$, since the regularity of a saturated homogeneous ideal $I$ equals the regularity of its sheafification.
The modules $M^i_X:=\oplus H^i(\mathcal I_X(t))$, $1\leq i \leq dim(X)$, are the {\it deficiency modules} of $X$.

We refer to \cite{Va} for definitions and results about Hilbert functions of standard graded algebras. In particular, according to \cite{Va}, for a standard graded $S$-algebra $M=\bigoplus M_t$ we let $H_M(t):=dim_K M_t$ be the {\it Hilbert function} of $M$, $\Delta H_M(t) := H_M(t) - H_M(t-1)$ for all $t>0$ and $\Delta H_M(0):=1$. In a similar way, the $i$-th difference function $\Delta^i H_M(t)$ is defined for each integer $i$ not higher than the Krull dimension $k+1$ of $M$.

For $t \gg 0$, we have $H_M(t)=P_M(t)$ where $P_M(z) \in \mathbb Q[z]$ is the {\it Hilbert polynomial} of $M$, which has degree $k$. The {\it regularity} of the Hilbert function $H_M$ is $\rho_M := \min\{\bar t \ \vert \ H_M(t)=P_M(t), \forall t\geq \bar t\}$. The Hilbert series $\sum_{t\in \mathbb N} H_M(t)z^t$ of $M$ is equal to a rational function ${h(z)}/{(1-z)^{k+1}}$, where $k+1$ is the Krull dimension of $M$ as stated above,  $h(z)=h_0+h_1z+\ldots+h_sz^s \in \mathbb Z[z]$ is the {\it $h$-polynomial} of $M$ and $(h_0,h_1,\ldots,h_s)$ is the {\it $h$-vector} of $M$. It is noteworthy that
$(h_0,h_1,\ldots,h_s)=(\Delta^{k+1} H_M(0),\Delta^{k+1} H_M(1),\ldots,$ $\Delta^{k+1} H_M(\rho_M+k))$. We say that an $h$-vector $(h_0,h_1,\ldots,h_s)$ is {\em positive} if for all $0\leq i \leq s$ the integer $h_i$ is positive.

Recall that, given two positive integers $a$ and $d$, $a$ can be written uniquely in the form
$$a=\binom{k(d)}{d} + \binom{k(d-1)}{d-1} + \ldots + \binom{k(j)}{j}$$
where $k(d)> k(d-1)>\ldots> k(j)\geq j\geq 1$. Let
$$a^{\langle d\rangle}:=\binom{k(d)+1}{d+1} + \binom{k(d-1)+1}{d} + \ldots + \binom{k(j)+1}{j+1}.$$
A numerical function $H:\mathbb N \rightarrow \mathbb N$ is {\it admissible} if $H(t)\leq H(t-1)^{<t-1>}$ for every $t\geq 2$ and $H(0)=1$. A finite sequence of positive integers $h_0,\ldots,h_s$ is admissible if the corresponding function given by $H(0)=h_0$, $H(1)=h_1$, $\ldots$, $H(s)=h_s$, $H(s+i)=0$, for every $i>0$, is admissible.

\begin{theorem}\label{Valla} {\rm \cite[Theorem 1.5]{Va}}
A finite sequence $h_0,h_{1,}\ldots,h_s$ of (positive) integers is the $h$-vector of a Cohen-Macaulay (standard) homogeneous $K$-algebra if and only if it is admissible.
\end{theorem}

\begin{example}\label{esempio_algebrico}
By Theorem \ref{Valla}, the $h$-vector of a Cohen-Macaulay homogeneous $K$-algebra is always positive.  Nevertheess, there are non-Cohen-Macaulay homogeneous $K$-algebras with positive but non-admissible $h$-vector. For example, the function $H(t): \ 1 \ 4 \ 10 \ P(t)=9t-10$ is the Hilbert function of a standard homogeneous $K$-algebra, with positive but non-admissible $h$-vector $(1,2,3,1,2)$. By \cite[Theorem 3.3]{GMR}, we can construct a reduced $K$-algebra $A$ with this property, because the first difference of $H(t)$ is admissible. We checked that $Proj(A)$ is a space projective curve with embedded components.
\end{example}

Let $X\subset {\mathbb P}^n_K$ be a closed subscheme of dimension $k$ and $I$ its defining (saturated) ideal. Instead of $H_{S/I}$, $P_{S/I}$, $\rho_{S/I}$ we can write $H_{X}$, $P_{X}$, $\rho_X$. Recall that $\sum_{t}\Delta^{k+1} H_X(t)=deg(X)$. If $X$ has positive $h$-vector, then $\Delta^{k+1} H_X(t)\geq 0$ \ for every $t$, and \ if $\Delta^{k+1} H_X(j)= 0$ for some $j$, then $\Delta^{k+1} H_X(t)= 0$ for every $t\geq j$.

A projective scheme $X\subset \mathbb P_K^n$ has \emph{Cohen-Macaulay postulation} if there is an aCM projective scheme $W\subset \mathbb P_K^n$ such that $H_X(t)=H_W(t)$ for every integer $t$.

\begin{remark} \label{aCM}
If $C$ is a curve with Cohen-Macaulay postulation, then $reg(C)>\rho_C+1$ \cite[Proposition 2.4]{CDG} and the $h$-vector of $C$ is positive.
\end{remark}

\begin{proposition}
The arithmetic genus $g$ of a projective curve $C$ with positive $h$-vector is non-negative.
\end{proposition}

\begin{proof}
If $(h_0,h_1,\ldots,h_s)$ is the $h$-vector of $C$ and $H_C$ is the Hilbert function of $C$, then the first difference $\Delta H_C=(h_0,h_0+h_1,\ldots,\sum_{0\leq i \leq s} h_i=deg(C),deg(C),\ldots)$ is strictly increasing until it becomes equal to $deg(C)$, because the $h$-vector is positive. Hence, by construction we get that $deg(C)(s-1)+1\geq\sum_{0\leq t \leq s-1} \Delta H_C(t)=H_C(s-1)=P_C(s-1)=deg(C)(s-1)+1-g$, so  $g\geq 0$.
\end{proof}

Let $k>0$ and $h\in S_1$ be a general linear form which is not a zero-divisor on $S/I$ and $J:=(I,h)$. Let $Z\subset{\mathbb P}^{n-1}_K$ be the scheme of dimension $k-1$ defined by the saturated ideal $J^{sat}/(h) = (I,h)^{sat}/(h)$, i.e. the general hyperplane section of $X$.

\begin{remark}\label{sezione}\rm
(a) By the short exact sequence
$$0 \longrightarrow (S/I)_{t-1} \buildrel{\cdot h} \over
\longrightarrow (S/I)_t \longrightarrow (S/J)_t\longrightarrow 0$$
it follows that $H_{S/J}(t) = \Delta H_{S/I}(t)$ and then
$\rho_{S/J}=\rho_X+1$, so that $\Delta H_X(t)=H_{S/J}(t)\geq H_Z(t)$ for every $t$ and $H_{S/J}(t)=H_Z(t)$ for $t\geq \max\{\rho_Z,\rho_{S/J}\}$.

(b) Recall that $X$ is aCM if and only if its deficiency modules are null (e.g. \cite[Lemma 1.2.3]{Mi}). Moreover, if $X$ is aCM, then $\Delta H_X(t)=H_Z(t)$ for all $t$, but the converse is false in general, but true for curves. Anyway, an equidimensional and lCM closed subscheme $X$ of dimension $k\geq 2$ is aCM if and only if its general hyperplane section is aCM (\cite[Proposition 2.1]{HU} or \cite[Theorem 1.3.3]{Mi}).
\end{remark}

It is always possible to find positive integers $\beta_1\leq\ldots\leq\beta_{n-k}$ and a complete intersection $Y$ (c.i. for short) of type $(\beta_1,\ldots,\beta_{n-k})$ containing $X$ (see Theorem 3.14 of Chapter VI of \cite{K}). Recall that a complete intersection $Y$ is also a Gorenstein scheme and that a Gorenstein scheme of codimension two is always a complete intersection (here Gorenstein means arithmetically Gorenstein). Moreover, recall that $reg(Y)=\sum \beta_i - (n-k)+1$.

We refer to \cite{Mi} for general results on liaison theory and, especially, on algebraically linked schemes. In particular, we recall that linkage is preserved by hypersurface section.

\begin{theorem}\label{DGO} {\rm \cite[Theorem 3]{DGO}}
Let $Y$ be a $k$-dimensional Gorenstein scheme containing properly a $k$-dimensional aCM scheme $Z$ defined by a saturated ideal $I(Z)$. Let $Z'$ be the scheme algebraically linked to $Z$ by $Y$, i.e. defined by $I(Z')=(I(Y):I(Z))$. Let $\bar\alpha$ and $\bar\alpha'$ be the initial degrees of $I(Z)/I(Y)$ and of $I(Z')/I(Y)$, respectively. Then,
\begin{itemize}
\item[(i)] $reg(Z)+\bar\alpha'=reg(Z')+\bar\alpha=reg(Y)$;
\item[(ii)] $\Delta^{k+1} H_Y(t)=\Delta^{k+1} H_Z(t)+\Delta^{k+1} H_{Z'}(reg(Y)-1-t)$, for every $0\leq t \leq reg(Y)-1$.
\end{itemize}
\end{theorem}

If $X\subset \mathbb P^n_K$ is an equidimensional and lCM subscheme of dimension $k$ (we refer to \cite{HIO} for results about equidimensional lCM closed subschemes), let $X_{k-i}$ be the subscheme obtained by applying $i$ successive general hyperplane sections to $X$, where $0\leq i\leq k$. In particular, if $k\geq 2$, $C:=X_{1}$ is the curve obtained by applying $k-1$ successive general hyperplane sections to $X$ and $Z:=X_0$ is a general hyperplane section of $C$.

We denote by $C'$ and $Z'$, respectively, the curve and the $0$-dimensional scheme linked to $C$ and $Z$ by general hyperplane sections of a complete intersection $Y$ containing properly $X$.

\begin{lemma}\label{Lemma}
The arithmetic genus of $C'$ is
\begin{equation}\label{eq:genere}
g'=D_t+deg(C')(-t+\sum\beta_i-(c+1)-1)+1,
\end{equation}
where $D_t=\Delta^{k-1} P_Y(t)-P_C(t)$, for every $t\geq\max\{\rho_C, reg(Y)-2\}$.
\end{lemma}

\begin{proof}
Denoting by $\bar g$ and $g$ the arithmetic genus respectively of $Y$ and $C$, we obtain
$$\Bigl(\prod_i \beta_i-deg(C')\Bigr)\cdot t+1-g = \prod_i \beta_i \cdot t+1-\bar g-D_t.$$
By applying \cite[Corollary 5.2.14]{Mi}, for which $g-g'=\frac{1}{2}(\sum_{i=1}^n\beta_i-n-1)(deg(C)-deg(C'))$, and by the shape of the arithmetic genus of a complete intersection curve (see, for example, \cite[page 36]{Mi}), for which $\bar g=\frac{1}{2} \prod_i\beta_i(\sum_{i=1}^n \beta_i-n-1)+1$, we have the thesis.
\end{proof}

If $i\neq k$, then $X_{k-i}$ is equidimensional and lCM as $X$, by the sequence
\begin{equation}\label{Castelnuovo}
0 \rightarrow H^0({\mathcal I}_X(t-1)) \rightarrow H^0({\mathcal I}_{X}(t)) \rightarrow H^0({\mathcal I}_{Z,H}(t)) \rightarrow
\end{equation}
$$\to H^1({\mathcal I}_X(t-1))\rightarrow H^1({\mathcal I}_{X}(t)) \to H^1({\mathcal I}_{Z}(t)) \to \ldots.
$$
Further, in codimension two, if $X$ is non-degenerate also $X_{k-i}$ is non-degenerate, for every $i<k$ \cite[Proposition 1.4]{CDG}.

\begin{lemma}\label{tk}
Letting $t_0:=\min\{t\in \mathbb N : H_X(t)<H_Y(t)\}$ and, for every $1\leq i\leq k$, $t_i:=\min\{t\in \mathbb N : H_{X_{k-i}}(t)<\Delta^i H_Y(t)\}$, we have
$$t_0\geq t_1\geq \ldots \geq t_k= \min\{t\in \mathbb N : \Delta H_{Z}(t)<\Delta^{k+1} H_Y(t)\}.$$
\end{lemma}

\begin{proof}
By Remark \ref{sezione}(a), we have $\Delta H_Y(t_0)>\Delta H_X(t_0)\geq H_{X_{k-1}}(t_0)$ from which we deduce that $t_1\leq t_0$; from $\Delta^2 H_Y(t_1)>\Delta H_{X_{k-1}}(t_1)\geq H_{X_{k-2}}(t_1)$ we deduce that $t_2\leq t_1$; and so on.
\end{proof}

Finally, we recall the notion of space extremal curve together with a characterization in terms of the $h$-vector.

\begin{definition} \cite[Corollary 6 and Definition 7]{E95}
Let $ch(K)=0$. A space curve $C$ of degree $d\geq 3$ and arithmetic genus $g$ is called {\em extremal} if it satisfies the following conditions:
\begin{itemize}
\item[(i)] $r_a:=\min\{t \ \vert \ h^1(\mathcal I_C(t))\not= 0\} = g+1-(d-2)(d-3)/2$;
\item[(ii)] $r_0:=\max\{t \ \vert \ h^1(\mathcal I_C(t))\not= 0\} = d(d-3)/2-g$;
\item[(iii)] $h^1(\mathcal I_C(t))=(d-2)(d-3)/2-g$, for every $r_a\leq t\leq r_0$.
\end{itemize}
\end{definition}

In \cite[Theorem 8]{E95} there is a first complete description of a space extremal curve of degree $d\geq 5$, that is then improved in \cite[Theorem 2.1]{CGN07}. A space extremal curve is  supported on the union of a line and a plane curve that can either meet at a point (aCM case) or are disjoint or have the same support. For a large class of explicit examples of extremal space curves see \cite{MDP96} and the references therein.

\begin{proposition}\label{1-2-1}
Let $ch(K)=0$ and $C\subset \mathbb P^3_K$ be a space curve of degree $d\geq 5$ with general hyperplane section $Z$. Then $C$ is an extremal curve if and only if $\Delta H_{Z}(t) \ : \ 1 \ 2 \ 1 \ \ldots \ 1$.
\end{proposition}

\begin{proof}
It is enough to apply \cite[Theorem 8]{E95}, because $Z$ has character $(d-1,2)$ iff $\Delta H_{Z}(t) \ : \ 1 \ 2 \ 1 \ \ldots \ 1$, by the definition of character of a plane $0$-dimensional scheme (see \cite{EP90}). One can also use geometric arguments due to \cite{Strano} (to find a plane curve $C_{\pi}$ of degree $d-1$ contained in $C$) and then \cite[Theorem 2.1]{CGN07}.
\end{proof}

%%%%%%%%%%%%%%%%%%%%%%%%%%%%%%%%%%%%%%%%%%%%%%%%%%%%%%
%% Cohen-Macaulayness conditions in codimension two %%
%%%%%%%%%%%%%%%%%%%%%%%%%%%%%%%%%%%%%%%%%%%%%%%%%%%%%%

\section{Cohen-Macaulayness conditions in codimension two}

With the notation stated in Section 1, let $X\subset \mathbb P^n_K$ be a non-degenerate codimension two subscheme that is lCM and equidimensional. Thus, $k=n-2$ is the dimension of $X$ and $C$ is a space curve. Recall that, if $i\neq k$, then $X_{k-i}$ is non-degenerate, equidimensional and lCM as $X$.
The $h$-vector of a complete intersection $Y\subset \mathbb P^n_K$ of type $(\beta_1,\beta_2)$ and of dimension $k=n-2$ is:
\begin{equation} \label{DeltaY}
\Delta^{k+1} H_Y(t): \ 1 \ \ 2 \ \ \ldots \ \beta_1 \ \ldots \ \beta_1 \ {\beta_1-1}\ \ldots \ \ 2 \ \ 1 \ \ 0,
\end{equation}
where $\beta_1$ appears $\beta_2-\beta_1+1$ times, and $reg(Y)=\beta_1+\beta_2-1$.

%%%%%%%%%%%%%%%%%%%%%%%%%%%%%%%%%%%%%%%%%%%%%%
%% Cohen-Macaulayness conditions for curves %%
%%%%%%%%%%%%%%%%%%%%%%%%%%%%%%%%%%%%%%%%%%%%%%

\subsection{Space curves in characteristic different from two}

Over every algebraically closed field of characteristic different from $2$, if $C$ is a space curve contained properly in a c.i. $Y$ with $deg(Y)-deg(C)\leq 5$, then $C$ is aCM if and only if $reg(C)>\rho_C+1$ \cite[Theorem 4.2]{CDG}. This result is sharp and holds also if we replace the condition $reg(C)>\rho_C+1$ by the stronger condition for which $C$ has Cohen-Macaualy postulation; indeed, to have Cohen-Macaualy postulation implies $reg(C)>\rho_C+1$, as we recall in Remark \ref{aCM}. Here, we prove an analogous statement with the condition that the $h$-vector of $C$ is positive. First, we give an example of a lCM space curve with positive but non-admissible $h$-vector.
Note that a curve $C$ without Cohen-Macaulay postulation has a non-admissible $h$-vector $(h_0,h_1,\ldots,h_s)$, but its {\it first sum} $(h_0,h_0+h_1,\ldots,\sum_i h_i,\ldots)$, that is equal to $\Delta H_C$, has to be admissible by Remark \ref{sezione}(a).

\begin{example} \label{ESEMPIO}
In Example \ref{esempio_algebrico} we highlight the existence of a space curve with positive but non-admissible $h$-vector. That curve has embedded components, so it is not lCM. Here we describe an example of lCM space curve that is non-aCM and has positive but non-admissible $h$-vector. Let $C$ be the rational curve given by the rational map
$\Phi: \mathbb P^1_K \rightarrow \mathbb P^3_K$ such that $\Phi (u,v):=(u^5+v^5, u^4v+uv^4,u^3v^2+v^5, u^2v^3)$.
The Hilbert function of $C$ is $H_C(t): 1 \ 4 \ 9 \ 5t+1$. We apply to $C$ three successive basic double links (we refer to \cite{Mi} for the definition of a basic double link) of type $(1,7)$, $(1,7)$ and $(1,9)$, respectively, obtaining a curve $\bar C$ with non-admissible $h$-vector $(1,2,3,4,5,5,5,1,2)$, by \cite[Proposition 5.4.5(d)]{Mi}.
\end{example}

\begin{theorem}\label{thmcurve}
Let $ch(K)\not=2$ and $C\subset \mathbb P^3_K$ be a space curve contained properly in a c.i. $Y$ with $deg(Y)-deg(C)\leq 5$. Then $C$ is aCM if and only if the $h$-vector of $C$ is positive.
\end{theorem}

\begin{proof}
We look for all the possible positive $h$-vectors of $C$, taking into account the relations among the integers $t_0$ and $t_1$ of Lemma \ref{tk} and the fact that, by Theorem \ref{DGO}, $\Delta H_Z$ can differ from the $h$-vector of $Y$ in degrees $\geq reg(Y)-deg(Z')$.
\par
When the positive $h$-vector of the space curve $C$ is admissible, $C$ has Cohen-Macaulay postulation by Theorem \ref{Valla} and then we get the thesis by \cite[Theorem 4.2]{CDG}.
%So, we are interested in positive but non-admissible $h$-vectors for $C$.
%
If $Z'$ is not degenerate or is degenerate of degree at most $4$, we find that all the possible positive $h$-vectors of $C$ are admissible.
\par
Let $Z'$ be degenerate of degree $5$, i.e. $\Delta H_{Z'}: 1\ 1 \ 1 \ 1 \ 1 \ 0$. If $ch(K)=0$ then $C'$ is a plane curve by \cite[Theorem 2.1]{H94} and so $C$ is aCM. If $ch(K)=p>0$, we will find that the arithmetic genus $g'$ of $C'$ is positive, in contradiction with \cite[proof of Theorem 3.3]{H94}, where Hartshorne studies non-degenerate space curves swith degenerate general hyperplane section.
By looking at all the possible sequences that can be $h$-vectors of $C$ in this case, we obtain that $\rho_C\leq reg(Y)-1$. So, by Lemma \ref{Lemma} applied with $t=reg(Y)-1=\beta_1+\beta_2-2\geq \rho_{C}$, we obtain $g'=D_{reg(Y)-1}-9$, where $D_{reg(Y)-1}$ is the difference between the values assumed on $reg(Y)-1$ by the Hilbert polynomials of $Y$ and of $C$, respectively. In the following table we collect all the sequences that give rise to Hilbert polynomials $P_C(t)$ for $C$ that assume the maximum possible value on $reg(Y)-1$, with the consequence that the corresponding value of $D_{reg(Y)-1}=11$ is the minimum possible. Hence, we obtain $g'\geq 11-9>0$.

\vskip 2mm
\begin{small}
\begin{center}
\begin{tabular}{c||c| c c c c c }
\hline
$\beta_1$ & $t$ & $reg(Y)-5$ & $reg(Y)-4$ & $reg(Y)-3$ & $reg(Y)-2$ & $reg(Y)-1$ \\
\hline
$2$ &$\Delta^{2} H_Y$ & $2$ & $2$ & $2$ & $2$ & $1$ \\
&$\Delta^{2} H_C$ & $2$ & $2$ & $1$ & $0$ & $0$ \\
&$\Delta H_Z$     & $1$ & $1$ & $1$ & $1$ & $0$ \\
\hline
$3$ &$\Delta^{2} H_Y$ & $3$ & $3$ & $3$ & $2$ & $1$ \\
&$\Delta^{2} H_C$ & $3$ & $3$ & $1$ & $0$ & $0$ \\
&$\Delta H_Z$     & $2$ & $2$ & $2$ & $1$ & $0$ \\
\hline
$4$ &$\Delta^{2} H_Y$ & $4$ & $4$ & $3$ & $2$ & $1$ \\
&$\Delta^{2} H_C$ & $4$ & $4$ & $1$ & $0$ & $0$ \\
&$\Delta H_Z$     & $3$ & $3$ & $2$ & $1$ & $0$ \\
\hline
$\geq 5$ &$\Delta^{2} H_Y$ & $5$ & $4$ & $3$ & $2$ & $1$ \\
&$\Delta^{2} H_C$ & $5$ & $4$ & $1$ & $0$ & $0$ \\
&$\Delta^2 H_C$ & $4$ & $6$ & $0$ & $0$ & $0$ \\
&$\Delta H_Z$     & $4$ & $3$ & $2$ & $1$ & $0$ \\
\hline
\end{tabular}
\end{center}
\end{small}
\vskip 2mm
\end{proof}

\subsection{Codimension two subschemes in characteristic zero}

In this subsection we suppose $ch(K)=0$ and consider codimension two lCM equidimensional subschemes in $\mathbb P^n_{K}$ with $n\geq 4$.

\begin{theorem}\label{d'=5}
Assuming $ch(K)=0$ and $n\geq 4$, let $X\subset \mathbb P^n_K$ be a lCM equidimensional codimension two subscheme contained in a complete intersection $Y$ with $deg(Y)-deg(X)\leq 5$. Then $X$ is aCM if and only if $X$ has positive $h$-vector.
\end{theorem}

\begin{proof}
For the cases $deg(Y)-deg(X)\leq 4$ we refer to the proof of \cite[Theorem 4.9]{CDG}, because in that proof we used only the positivity of the $h$-vector.

For the case $deg(Y)-deg(X)=5$, according to the notation of Section 1, by Remark \ref{sezione}(b) we have that $\Delta H_C(t)\geq H_Z(t)$ for every $t$ and the equality holds for every $t$ iff $C$ is aCM.  In the hypothesys that the $h$-vector of $X$ is positive, we consider all possible $\Delta H_{Z'}(t)$ such that $deg(Z')=5$.

If $\Delta H_{Z'}: \ 1 \ 1 \ 1 \ 1 \ 1$, then $Z'$ is degenerate with $deg(C')=deg(Z')=5\geq 3$; so, $C'$ is a plane curve, because $ch(K)=0$ \cite[Theorem 2.1]{H94}.

If $\Delta H_{Z'}(t):\ 1 \ 2 \ 2 $, then applying Lemma \ref{tk} we have $t_i\geq t_k=reg(Y)-3$ for every $0\leq i < k$, by Theorem \ref{DGO}, and
$\Delta^{k+1-i}H_{X_{k-i}}(t)=\Delta^{k+1} H_Y(t), \text{ for every } t< reg(Y)-3 \text{ and } 0\leq i\leq k.$
Moreover, we have $\sum_{t}\Delta^{k+1} H_X(t)=deg(X)=deg(Z)=\sum_{t\leq reg(Y)-3} \Delta H_Z(t)$, and  $\Delta^{k+1}H_X(reg(Y)-2)\geq \Delta H_Z(reg(Y)-2)$; so, it follows
$$
\Delta H_Z(t)=\Delta ^{k+1} H_X(t), \text{ for every } t\geq reg(Y)-3,
$$
because the $h$-vector is positive, and we can apply Remark \ref{sezione}(b).

If $\Delta H_{Z'}(t):\ 1 \ 2 \ 1 \ 1 $, then the curve $C'$ is extremal by Proposition \ref{1-2-1}. Thus, $C'$ is aCM or there is not a lCM surface with $C'$ as general hyperplane section, by \cite[Theorem 1.1 and Corollary 3.6]{CGN02}.
\end{proof}

By the same arguments applied in the proof of Theorem \ref{d'=5} we get the following.

\begin{proposition}\label{varie}
The statement of Theorem \ref{d'=5} holds also for every $d'=deg(Y)-deg(X)\geq 6$ and  $Z'$ such that:
\begin{itemize}
\item[(i)] $Z'$ is degenerate;
\item[(ii)] $Z'$ has maximal rank, that is $H_{Z'}(t)=\min\{d',\binom{t+2}{2}\}$.
\end{itemize}
Moreover, in $\mathbb P^n_K$, $n\geq 4$, there is not a lCM subscheme $X$ of codimension two with $Z'$ as $0$-dimensional hyperplane section such that $\Delta H_{Z'}(t) \ : \ 1 \ 2 \ 1 \ \ldots \ 1$.
\end{proposition}

\begin{remark} 
This result shows that the condition $deg(Y)-deg(X)\leq 5$ is not characterizing for curves forced to be aCM by the positivity of their $h$-vector.
\end{remark}

In \cite[Remark 4.10]{CDG} we gave an example of a non-aCM, but lCM equidimensional surface $X$ in $\mathbb P^4_K$ with Cohen-Macaulay postulation and $deg(Y)-deg(X)=10$, by exploiting the technique of Davis to construct examples of non-aCM space curves $C'$ of degree $d'\geq 6$ (see \cite{D} and \cite[Appendix]{CDG}) linked to curves with Cohen-Macaulay postulation.
More precisely, we applied an odd number of suitable liaisons to the Veronese surface $V$ in $\mathbb P^4_K$ (e.g. \cite[Cap. II, Ex. 7.7]{H}) for which we know that $h^1(\mathcal I_{V}(t))=0$ if $t\not=1$, $h^1(\mathcal I_{V}(1))=1$ and $h^2(\mathcal I_{V}(t))=0$ for every $t$ \cite[Example 3.7]{BMR} (see also \cite{Ro} for a study of the Veronese surface and its degenerations in an analogous topic). Starting from this example, now we exhibit a class of non-aCM surfaces in $\mathbb P^4_K$ with Cohen-Macaulay postulation.

As recalled in \cite[Appendix]{CDG}, by applying successive suitable liaisons to the union of two skew lines in $\mathbb P^3_K$, Davis constructs the following two types of space curves (where $Z$ is the general hyperplane section):
\begin{itemize}
\item[(1)] curves $D\subset \mathbb P^3_K$ of type $[a,r]$, with $a>r>0$, such that
$$\alpha(I(Z))=\rho_Z-1=a; \quad \Delta H_Z(a)=a; \quad \Delta H_Z(a+1)=r; \quad h^1(\mathcal I_D(a))=1;$$
\item[(2)] curves $D\subset \mathbb P^3_K$ of type $[[a,r]]$, with $a>r>1$, such that
$$\alpha(I(Z))=\rho_Z-1=a; \quad \Delta H_Z(a)=r; \quad \Delta H_Z(a+1)=1; \quad h^1(\mathcal I_D(a))=1.$$
\end{itemize}
For every $d'\geq 6$, $d'\not=7,8,12$, let $e_{d'}:=\max\left\{t: \binom{t}{2}\leq d'\right\} \ \text{ and } \ f_{d'}:=d'-\binom{e_{d'}}{2}$.
A non-aCM curve $C_{d'}$ with Cohen-Macaulay postulation and $deg(Y)-deg(C_{d'})=d'$ is obtained by applying a liaison with c.i. of type $(\beta_1,\beta_2)=(e_{d'}+1,e_{d'}+1)$ to a curve $C'_{d'}$, chosen among the curves $D$ of type $[a,r]$ or $[[a,r]]$ in accordance with the table below:
\vskip 2mm
\begin{small}
\begin{center}
 \begin{tabular}{|c|c|}
   \hline
 $e_{d'}-f_{d'}$&  $C'_{d'}$  \\
   \hline
1 & $[[e_{d'}-1,f_{d'}-1]]$ \\
2 & $[[e_{d'}-1,2]]$   \\
3 &  $[[e_{d'}-1,f_{d'}]]$  \\
$\geq 4$ & $[e_{d'}-2,f_{d'}+1]$ \\
   \hline
 \end{tabular}
\end{center}
\end{small}
\vskip 2mm
For $d'=7,8,12$, the curves $C_{d'}$ are constructed in a slightly different way.

\begin{proposition}\label{counterexamples}
There exists a lCM equidimensional surface $X\subset\mathbb P^4_K$ contained in a complete intersection $Y$, with $deg(Y)-deg(X)=d'$, such that $X$ is non-aCM, but has Cohen-Macaulay postulation, for every integer
\begin{small}
$$d'\in \{10,14,15,19,20,21,22\}\cup \Bigl(\cup_{t\geq 8}\Bigl\{\binom{t}{2}-3, \binom{t}{2}-2, \binom{t}{2}-1, \binom{t}{2}, \binom{t}{2}+1 \Bigr\}\Bigr).$$
\end{small}
\end{proposition}

\begin{proof}
It is crucial for our purpose that the general hyperplane section of the Veronese surface $V$ of $\mathbb P^4_K$ is a curve $\mathcal C$ that is linked to two skew lines by a liaison of type $(2,3)$, as one can check by a computation and by \cite[Lemma-Definition, page F8]{D}. Indeed, many of the curves of types $[a,r]$ and $[[a,r]]$ are constructed from two skew lines by liaisons, the first of which generates $\mathcal C$. Thanks to the fact that linkage preserves general hyperplane sections, by applying to $V$ the same liaisons, we obtain surfaces $X'_{d'}$ whose general hyperplane sections are the curves of type $[a,r]$ or $[[a,r]]$. With a further liaison of type $(\beta_1,\beta_2)=(e_{d'}+1,e_{d'}+1)$, we obtain a surface $X$ whose general hyperplane section is the curve $C_{d'}$ of Davis, which has Cohen-Macaulay postulation. Moreover, if the number of applied liaisons is odd, by the Hartshorne-Schenzel Theorem we obtain that $h^1(\mathcal I_{X}(t))= h^2(\mathcal I_{V}!
 (t))=0$, for every $t$, so that $\Delta H_{X}(t) =H_{C_{d'}}(t)$ \cite[Remark 2.1.3]{Mi} and also $X$ has Cohen-Macaulay postulation, because $X$ shares its $h$-vector with $C_{d'}$.

First, we observe that the described strategy works in the cases considered in the following table: we apply successively liaisons of the listed types to the surface of $\mathbb P^4_K$ denoted by $S$, obtaining the above surface $X'_{d'}$ of degree $d'$; after a further liaison, we get the surface $X$ with Cohen-Macaulay postulation.
\vskip 2mm
\begin{small}
\begin{center}
\begin{tabular}{l| l l l l }
\hline
$d'$ & $S$ & liaisons & further liaison & $h$-vector of $X$ \\
\hline
$10$ & $V$ & $(3,4),(3,6)$ & $(6,6)$ & $[1, 2, 3, 4, 5, 6, 5]$ \\
$14$ & $X'_{10}$ & $(4,6),(4,7)$ & $(6,6)$ & $[1, 2, 3, 4, 5, 6, 1]$ \\
$15$ & $V$ & $(3,3),(3,4),(4,5),(4,7)$ & $(7,7)$  &  $[1, 2, 3, 4, 5, 6, 7, 6]$ \\
$19$ & $X'_{14}$ & $(5,7),(5,8)$ & $(7,7)$ & $[1, 2, 3, 4, 5, 6, 7, 2]$ \\
$20$ & $X'_{15}$ & $(5,7),(5,8)$ & $(7,7)$ & $[1, 2, 3, 4, 5, 6, 7, 1]$ \\
$21$ & $V$ & $(3,3),(3,4),(4,4),(4,5),(5,6)(5,8)$ & $(8,8)$ & $[1, 2, 3, 4, 5, 6, 7, 8, 7]$ \\
$22$ & $V$ & $(3,4),(4,5),(5,6),(5,8)$ & $(8,8)$ & $[1,2,3,4,5,6,7,8,6]$ \\
%$25$ & $X'_{19}$ & $(6,8),(6,9)$ & $(8,8)$ & $[1, 2, 3, 4, 5, 6, 7, 8, 3]$ \\
%$26$ & $X'_{20}$ & $(6,8),(6,9)$ & $(8,8)$ & $[1, 2, 3, 4, 5, 6, 7, 8, 2]$ \\
%$27$ & $X'_{21}$ & $(6,8),(6,9)$ & $(8,8)$ & $[1, 2, 3, 4, 5, 6, 7, 8, 1]$ \\
%$28$ & $V$ & $(3,3),(3,4),(4,4),(4,5),$ &  & \\
%     &     & $(5,5),(5,6),(6,7),(6,9)$ & $(9,9)$ & $[1, 2, 3, 4, 5, 6, 7, 8, 9, 8]$ \\
\hline
\end{tabular}
\end{center}
\end{small}
\vskip 2mm
In this construction there is a type of recursion useful to prove that our strategy works well also in the remaining cases. Indeed, for $d'\in \cup_{t\geq 8}\Bigl\{\binom{t}{2}-3, \binom{t}{2}-2, \binom{t}{2}-1\Bigr\}$, by the construction of \cite{D} it is enough to apply successively two liaisons of type $(t-2,t)$ and $(t-2,t+1)$ to $X'_{\bar d}$, where $\bar d=\binom{t-1}{2}-2,\binom{t-1}{2}-1,\binom{t-1}{2}$, respectively. With a further liaison of type $(t,t)$ we obtain the desidered surface $X$, after a total odd number of liaisons applied to $V$.

For $d'=\binom{t}{2}$, we observe that the curve $C'_{d'}$ is of type $[t-2,1]$ and is obtained from $\mathcal C$ by liaisons of type $(3,3),(3,4),\ldots,(t-3,t-3),(t-3,t-2),(t-2,t-1),(t-2,t+1)$. Starting from the Veronese surface $V$, with a further liaison of type $(t+1,t+1)$ we obtain the desidered surface $X$.

For $d'=\binom{t}{2}+1$, the curve $C'_{d'}$ is of type $[t-2,2]$ and is obtained from $\mathcal C$ by liaisons of two types: first $(3,3),(3,4),\ldots,(t-5,t-5),(t-5,t-4)$ and then $(t-4,t-3)$, $(t-3,t-2)$, $(t-2,t-1)$, $(t-2,t+1)$. Starting from the Veronese surface $V$, with a further liaison of type $(t+1,t+1)$ we obtain the desidered surface.
\end{proof}

\begin{remark}
In \cite{MN}, the authors give algorithms to construct codimension $2$ non-aCM subschemes with Cohen-Macaulay postulation. As already observed in the Introduction, the difference with the result of Proposition \ref{counterexamples} is that in \cite{MN} the authors use basic double links, here we use an odd number of liaisons.
\end{remark}

\section{Cohen-Macaulayness conditions in codimension higher than two}

With the same notation of the previous sections, $Y$ is a complete intersection of type $(\beta_1,\ldots, \beta_{n-k})$ containing properly $X$ and $reg(Y)=\sum \beta_i-c+1$, where $c=n-k$ is the codimension of $X$ and $Y$.

In this section we use the notion of Borel ideal and suppose that $ch(K)=0$. Recall that a Borel ideal is an ideal fixed under the action of the Borel subgroup of upper-triangular invertibles matrices, if $x_0<x_1<\ldots <x_n$, or under the action of the lower-triangular invertibles matrices, if $x_0>x_{1}>\ldots >x_n$. Here, we consider the latter setting. In generic coordinates, the initial ideal of an ideal I, with respect to a fixed term order $\prec$, is a constant Borel-fixed monomial ideal called the \emph{generic initial ideal} of I. We denote by $gin(I)$ the generic initial ideal of a homogeneous ideal $I$ with respect to the degree reverse lexicographic term order and we set $gin(X):=gin(I(X))$ for any subscheme $X$. For a survey on this subject we refer to \cite{Gr}. We will need the following properties of $gin(X)$.

\begin{theorem} \label{borel} {\rm (\cite[Theorem 4.2]{BS}, \cite{Gr})} If $X\subset \mathbb P^n_K$ is a closed subscheme, then $reg(X)=reg(gin(X))$; if moreover $Z$ is the general hyperplane section of $X$, then $gin(Z)={(gin(X),x_n)^{sat}}/{(x_n)}$.
\end{theorem}

\begin{theorem}\label{c>2}
Let $X\subset \mathbb P^n_K$ ($n\geq 4$) be a lCM equidimensional subscheme of codimension $c=n-k\geq 3$ contained in a complete intersection $Y$ with $deg(Y)-deg(X)\leq 3$. Then $X$ is aCM if and only if $X$ has positive $h$-vector.
\end{theorem}

\begin{proof}
We will follow the same approach of the proof of Theorem 2.3.

If $\Delta H_{Z'}$ is either $1 \ 0$ or $1 \ 1 \ 0$ or $1 \ 2 \ 0$, then $reg(Y)-2 \leq t_k$, by the definition of $t_k$ in Lemma \ref{tk} and by Theorem \ref{DGO}. So, we have $\Delta H_Z(reg(Y)-2)\leq \Delta^{k+1} H_X(reg(Y)-2)\leq \Delta^{k+1} H_Y(reg(Y)-2)$ and the thesis follows.

It remains to analyze the case $\Delta H_{Z'}: \ 1 \ 1 \ 1$, in which $reg(Y)-3 \leq t_k$. Recall that the arithmetic genus of a lCM curve of degree $3$ is $\leq 1$ and the equality holds if and only if the curve is planar (e.g. \cite{H94}).

Let $c=3$ and $X$ be non-aCM. Then, by Theorem \ref{DGO} and by degree arguments, i.e. $\sum_{t\geq 0} \Delta^{k+1}H_X(t)=deg(X)=deg(Y)-3$, we obtain the following situation:
\vskip 2mm
\begin{small}
\begin{center}
\begin{tabular}{c| c c c c c c }
\hline
$t$ & $0$ & $\ldots$ & $reg(Y)-3$ & $reg(Y)-2$ & $reg(Y)-1$ & $reg(Y)$ \\
\hline
$\Delta^{k+1} H_Y$ & $1$ & $\ldots$ & $a_{reg(Y)-3}$ & $3$ & $1$ & $0$ \\
$\Delta^{k+1} H_X$ & $1$ & $\ldots$ & $a_{reg(Y)-3}$ & $1$ & $0$ & $0$ \\
$\Delta H_Z$       & $1$ & $\ldots$ & $a_{reg(Y)-3}-1$ & $2$ & $0$ & $0$ \\
\hline
\end{tabular}
\end{center}
\end{small}
\vskip 2mm
and $\Delta^{k+1} H_X(t) = \Delta^{k+1} H_Y(t)=0$, for every $t\geq reg(Y)=\sum \beta_i-c+1$.
Applying Lemma \ref{Lemma} with $t=reg(Y)-2=\sum \beta_i-c-1$ and computing $D_{t}=\Delta^{k-1} P_Y(t)-P_C(t)=2$, we obtain $g'=0$, hence $P_{C'}(t)=3t+1$. Using the already cited applet \textsc{BorelGenerator} of P. Lella, $gin(C')$ can be one of the following Borel ideals:
$$J_1 = (x_0,x_1,x_2^4,x_2^3x_3), \quad J_2 = (x_0,x_1^2,x_1x_2,x_1x_3,x_2^3), \quad J_3=(x_0,x_1^2,x_1x_2,x_2^2).$$
Since $Z'$ is a planar scheme, then $gin(Z')$ is univocally determined by its $h$-vector; so,  $gin(Z')=(x_0,x_1,x_2^3)$ and by Theorem \ref{borel} we exclude $J_3$. We exclude also $J_1$ because in that case $g'=1\not=0$, as $C'$ is a lCM planar curve of degree $3$.
Finally, we exclude $J_2$, because in that case $C'$ should be a non-degenerate space curve of degree $3$, meanwhile $Z'$ is degenerate in $\mathbb P^2_K$, contrary to \cite[Theorem 2.1]{H94}. Hence, we obtain the thesis for $c=3$.

Let $c\geq 4$ and $X$ be non-aCM. Then, we obtain several possible situations. The first one generalizes the case $c=3$ and is
\vskip 2mm
\begin{small}
\begin{center}
\begin{tabular}{c| c c c c c c }
\hline
$t$ & $0$ & $\ldots$ & $reg(Y)-3$ & $reg(Y)-2$ & $reg(Y)-1$ & $reg(Y)$ \\
\hline
$\Delta^{k+1} H_Y$ & $1$ & $\ldots$ & $a_{reg(Y)-3}$ & $c$ & $1$ & $0$ \\
$\Delta^{k+1} H_X$ & $1$ & $\ldots$ & $a_{reg(Y)-3}$ & $c-2$ & $0$ & $0$ \\
$\Delta H_Z$       & $1$ & $\ldots$ & $a_{reg(Y)-3}-1$ & $c-1$ & $0$ & $0$ \\
\hline
\end{tabular}
\end{center}
\end{small}
\vskip 2mm
the second one is
\begin{small}
\begin{center}
\begin{tabular}{c| c c c c c c }
\hline
$t$ & $0$ & $\ldots$ & $reg(Y)-3$ & $reg(Y)-2$ & $reg(Y)-1$ & $reg(Y)$ \\
\hline
$\Delta^{k+1} H_Y$ & $1$ & $\ldots$ & $a_{reg(Y)-3}$ & $c$ & $1$ & $0$ \\
$\Delta^{k+1} H_X$ & $1$ & $\ldots$ & $a_{reg(Y)-3}$ & $c-3$ & $1$ & $0$ \\
$\Delta H_Z$       & $1$ & $\ldots$ & $a_{reg(Y)-3}-1$ & $c-1$ & $0$ & $0$ \\
\hline
\end{tabular}
\end{center}
\end{small}
\vskip 2mm
where $\Delta^{k+1} H_X(t) = \Delta^{k+1} H_Y(t)=0$, for every $t\geq reg(Y)=\sum \beta_i-c+1$; only for $c>4$, there are other possible cases, in which we have always $\Delta^{k+1} H_X(reg(Y)-2)<c-3$.

In the first case, as for $c=3$, we obtain $g'=0$; in the second case, for $t=reg(Y)-2$, we compute $D_{t}=\Delta^{k-1} P_Y(t)-P_C(t)=3$ and obtain $g'=1$.
So, we have either $P_{C'}(t)=3t+1$ or $P_{C'}(t)=3t$. In the first case, as for $c=3$, we obtain for $C'$ the possible generic initial ideals \hskip 2mm $J'_1 = (x_0, \ldots,x_{n-3},x_{n-2}^4,x_{n-2}^3x_{n-1})$, \hskip 2mm
$J'_2 = (x_0,\ldots,x_{n-4},x_{n-3}^2,x_{n-3}x_{n-2},x_{n-3}x_{n-1},x_{n-2}^3)$ \hskip 2mm and \hskip 2mm $J'_3=(x_0,\ldots,x_{n-4},x_{n-3}^2,$ $x_{n-3}x_{n-2},x_{n-2}^2)$, that we exclude with the same arguments as before. In the second case, we get $gin(C')=(x_0,\ldots,x_{n-3},x_{n-2}^3)$ and $C'$ would be aCM.

In the other cases, applying Lemma \ref{Lemma} with $t=\rho_C\geq reg(Y)-2$, we obtain $g'>1$, that is absurd.
\end{proof}

\begin{example}
Just to give an example of the situations that can occur for $c>4$ in the proof of Theorem \ref{c>2}, we consider the following case
\begin{small}
\begin{center}
\begin{tabular}{c| c c c c c c }
\hline
$t$ & $0$ & $\ldots$ & $reg(Y)-3$ & $reg(Y)-2$ & $reg(Y)-1$ & $reg(Y)$ \\
\hline
$\Delta^{k+1} H_Y$ & $1$ & $\ldots$ & $a_{reg(Y)-3}$ & $c$ & $1$ & $0$ \\
$\Delta^{k+1} H_X$ & $1$ & $\ldots$ & $a_{reg(Y)-3}$ & $c-4$ & $1$ & $1$ \\
$\Delta H_Z$       & $1$ & $\ldots$ & $a_{reg(Y)-3}-1$ & $c-1$ & $0$ & $0$ \\
\hline
\end{tabular}
\end{center}
\end{small}
\vskip 2mm
For $t=reg(Y)-1$, we obtain $D_{t}=\Delta^{k-1} P_Y(t)-P_C(t)=8$ and $g'=3$.
\end{example}

\begin{proposition} \label{prop:c3}
The statement of Theorem \ref{c>2} holds also with $deg(Y)-deg(X)\geq 4$ if $H_{Z'}(t)=\min\{d',\binom{t+c}{c}\}$, i.e. $Z'$ has maximal rank in $\mathbb P^c_K$.
\end{proposition}

\begin{remark} \label{rm:curve}
In \cite[Proposition 4.6]{CDG} we prove that, in a projective space $\mathbb P_K^n$ of dimension $n\geq 4$ over a field $K$ of every characteristic, a curve $C \subset \mathbb P_K^n$ with $deg(Y)-deg(C)\leq 3$ is aCM if and only if $reg(C) > \rho_C+1$. Moreover, we show that this result is sharp.
\end{remark}

\bibliographystyle{amsplain}
\bibliography{pscci}

\providecommand{\bysame}{\leavevmode\hbox to3em{\hrulefill}\thinspace}
\providecommand{\MR}{\relax\ifhmode\unskip\space\fi MR }
\providecommand{\MRhref}[2]{%
  \href{http://www.ams.org/mathscinet-getitem?mr=#1}{#2}
}
\providecommand{\href}[2]{#2}
\begin{thebibliography}{10}

\bibitem{BS}
David Bayer and Michael Stillman, \emph{A criterion for detecting
  {$m$}-regularity}, Invent. Math. \textbf{87} (1987), no.~1, 1--11.

\bibitem{BMR}
Giorgio Bolondi and Rosa~M. Mir{\'o}-Roig, \emph{Two-codimensional {B}uchsbaum
  subschemes of {${\bf P}^n$} via their hyperplane sections}, Comm. Algebra
  \textbf{17} (1989), no.~8, 1989--2016.

\bibitem{CGN02}
Nadia Chiarli, Silvio Greco, and Uwe Nagel, \emph{Surfaces in {$\Bbb P^4$} with
  extremal general hyperplane section}, J. Algebra \textbf{257} (2002), no.~1,
  65--87.

\bibitem{CGN07}
\bysame, \emph{Families of space curves with large cohomology}, J. Algebra
  \textbf{307} (2007), no.~2, 704--726.

\bibitem{CMR}
F.~Cioffi, M.~G. Marinari, and L.~Ramella, \emph{Regularity bounds by minimal
  generators and {H}ilbert function}, Collect. Math. \textbf{60} (2009), no.~1,
  89--100.

\bibitem{CDG}
Francesca Cioffi and Roberta Di~Gennaro, \emph{Liaison and
  {C}ohen-{M}acaulayness conditions}, Collect. Math. \textbf{62} (2011), no.~2,
  173--186.

\bibitem{CLMR}
Francesca Cioffi, Paolo Lella, Maria~Grazia Marinari, and Margherita Roggero,
  \emph{Segments and {H}ilbert schemes of points}, Discrete Mathematics
  \textbf{311} (2011), 2238--2252.

\bibitem{DGO}
E.~D. Davis, A.~V. Geramita, and F.~Orecchia, \emph{Gorenstein algebras and the
  {C}ayley-{B}acharach theorem}, Proc. Amer. Math. Soc. \textbf{93} (1985),
  no.~4, 593--597.

\bibitem{D}
Edward~D. Davis, \emph{Curves which are close to complete intersections}, The
  {C}urves {S}eminar at {Q}ueen's, {V}ol.\ {VII} ({K}ingston, {ON}, 1990),
  Queen's Papers in Pure and Appl. Math., vol.~85, Queen's Univ., Kingston, ON,
  1990, pp.~Exp.\ No.\ F, 14.

\bibitem{DGM}
Edward~D. Davis, Anthony~V. Geramita, and Paolo Maroscia, \emph{Perfect
  homogeneous ideals: {D}ubreil's theorems revisited}, Bull. Sci. Math. (2)
  \textbf{108} (1984), no.~2, 143--185.

\bibitem{dQR05}
Victoria~E. de~Quehen and Leslie~G. Roberts, \emph{Non-{C}ohen-{M}acaulay
  projective monomial curves with positive {$h$}-vector}, Canad. Math. Bull.
  \textbf{48} (2005), no.~2, 203--210.

\bibitem{E95}
Ph. Ellia, \emph{On the cohomology of projective space curves}, Boll. Un. Mat.
  Ital. A (7) \textbf{9} (1995), no.~3, 593--607.

\bibitem{EP90}
Ph. Ellia and Ch. Peskine, \emph{Groupes de points de {${\bold P}^2$}:
  caract\`ere et position uniforme}, Algebraic geometry ({L}'{A}quila, 1988),
  Lecture Notes in Math., vol. 1417, Springer, Berlin, 1990, pp.~111--116.

\bibitem{GMR}
A.~V. Geramita, P.~Maroscia, and L.~G. Roberts, \emph{The {H}ilbert function of
  a reduced {$k$}-algebra}, J. London Math. Soc. (2) \textbf{28} (1983), no.~3,
  443--452.

\bibitem{GPSY}
A.~Goodarzi, M.~R. Pournaki, S.~A. Seyed~Fakhari, and S.~Yassemi, \emph{On the
  {$h$}-vector of a simplicial complex with {S}erre's condition}, J. Pure Appl.
  Algebra \textbf{216} (2012), no.~1, 91--94.

\bibitem{Gr}
Mark~L. Green, \emph{Generic initial ideals}, Six lectures on commutative
  algebra, Mod. Birkh\"auser Class., Birkh\"auser Verlag, Basel, 2010,
  pp.~119--186.

\bibitem{H}
Robin Hartshorne, \emph{Algebraic geometry}, Springer-Verlag, New York, 1977,
  Graduate Texts in Mathematics, No. 52.

\bibitem{H94}
\bysame, \emph{The genus of space curves}, Ann. Univ. Ferrara Sez. VII (N.S.)
  \textbf{40} (1994), 207--223 (1996).

\bibitem{HIO}
M.~Herrmann, S.~Ikeda, and U.~Orbanz, \emph{Equimultiplicity and blowing up},
  Springer-Verlag, Berlin, 1988, An algebraic study, With an appendix by B.
  Moonen.

\bibitem{HU}
Craig Huneke and Bernd Ulrich, \emph{General hyperplane sections of algebraic
  varieties}, J. Algebraic Geom. \textbf{2} (1993), no.~3, 487--505.

\bibitem{K}
Ernst Kunz, \emph{Introduction to commutative algebra and algebraic geometry},
  Birkh\"auser Boston Inc., Boston, MA, 1985, Translated from the German by
  Michael Ackerman, With a preface by David Mumford.

\bibitem{L}
Paolo Lella, \emph{{An efficient implementation of the algorithm computing the
  Borel-fixed points of a Hilbert scheme}}, ISSAC 2012-Proceedings of the 37th
  International Symposium on Symbolic and Algebraic Computation, ACM, New York,
  2012, pp.~242--248.

\bibitem{MDP96}
Mireille Martin-Deschamps and Daniel Perrin, \emph{Le sch\'ema de {H}ilbert des
  courbes gauches localement {C}ohen-{M}acaulay n'est (presque) jamais
  r\'eduit}, Ann. Sci. \'Ecole Norm. Sup. (4) \textbf{29} (1996), no.~6,
  757--785.

\bibitem{MN}
J.~C. Migliore and U.~Nagel, \emph{Numerical macaulification}, Clay volume in
  honor of Joe Harris (to appear), Available at arXiv:1202.2275, 2012.

\bibitem{Mi}
Juan~C. Migliore, \emph{Introduction to liaison theory and deficiency modules},
  Progress in Mathematics, vol. 165, Birkh\"auser Boston Inc., Boston, MA,
  1998.

\bibitem{MoNa}
Dannis Moore and Uwe Nagel, \emph{Algorithms for strongly stable ideals},
  Available at http://arxiv.org/abs/1110.4080, 2011.

\bibitem{Ro}
Margherita Roggero, \emph{Laudal-type theorems in {${\bf P}^N$}}, Indag. Math.
  (N.S.) \textbf{14} (2003), no.~2, 249--262.

\bibitem{RTV}
Maria~Evelina Rossi, Ng{\^o}~Vi{\^e}t Trung, and Giuseppe Valla,
  \emph{Castelnuovo-{M}umford regularity and finiteness of {H}ilbert
  functions}, Commutative algebra, Lect. Notes Pure Appl. Math., vol. 244,
  Chapman \& Hall/CRC, Boca Raton, FL, 2006, pp.~193--209.

\bibitem{Strano}
Rosario Strano, \emph{Curves and their hyperplane sections}, J. Pure Appl.
  Algebra \textbf{152} (2000), no.~1-3, 337--341, Commutative algebra,
  homological algebra and representation theory (Catania/Genoa/Rome, 1998).

\bibitem{Va}
Giuseppe Valla, \emph{Problems and results on {H}ilbert functions of graded
  algebras}, Six lectures on commutative algebra ({B}ellaterra, 1996), Progr.
  Math., vol. 166, Birkh\"auser, Basel, 1998, pp.~293--344.

\end{thebibliography}

\end{document}